\documentclass[12pt,a4paper]{article}
\usepackage{graphicx}
\usepackage[]{amsmath,amssymb,amsthm}
\usepackage{paralist}
\usepackage{epstopdf}

\newtheorem{theorem}{Theorem}[section]
\newtheorem{proposition}[theorem]{Proposition}
\newtheorem{lemma}[theorem]{Lemma}
\newtheorem{corollary}[theorem]{Corollary}
\theoremstyle{definition}
\newtheorem{definition}[theorem]{Definition}
\theoremstyle{remark}
\newtheorem{remark}[theorem]{Remark}

\newcommand{\R}{\mathbb{R}}
\newcommand{\N}{\mathbb{N}}
\newcommand{\Z}{\mathbb{Z}}
\newcommand{\Fix}{\textnormal{Fix}}
\newcommand{\inn}{\textnormal{in}}
\newcommand{\out}{\textnormal{out}}
\newcommand{\E}{{\mathcal E}}

\begin{document}

\begin{center}
{\large{\bf Switching in heteroclinic networks}}\\
\mbox{} \\
\begin{tabular}{cc}
{\bf Sofia B.\ S.\ D.\ Castro$^{\dagger}$} & {\bf Alexander Lohse$^{\ddagger,*}$} \\
{\small sdcastro@fep.up.pt} & {\small alexander.lohse@math.uni-hamburg.de}
\end{tabular}

\end{center}

\noindent $^{*}$ Corresponding author.

\noindent $^{\dagger}$ Faculdade de Economia and Centro de Matem\'atica, Universidade do Porto, Rua Dr.\ Roberto Frias, 4200-464 Porto, Portugal.

\noindent $^{\ddagger}$ Centro de Matem\'atica da Universidade do Porto, Rua do Campo Alegre 687, 4169-007 Porto, Portugal\footnote{Long-term address: Fachbereich Mathematik, Universit\"at Hamburg, Bundesstra{\ss}e 55, 20146 Hamburg, Germany}

\begin{abstract}
We study the dynamics near heteroclinic networks for which all eigenvalues of the linearization at the equilibria are real. A common connection and an assumption on the geometry of its incoming and outgoing directions exclude even the weakest forms of switching (i.e.\ along this connection). The form of the global transition maps, and thus the type of the heteroclinic cycle, plays a crucial role in this. We look at two examples in $\R^5$, the House and Bowtie networks, to illustrate complex dynamics that may occur when either of these conditions is broken. For the House network, there is switching along the common connection, while for the Bowtie network we find switching along a cycle.
\end{abstract}


\vspace{.3cm}


\section{Introduction}
In dynamical systems a heteroclinic cycle is a set of finitely many equilibria and trajectories connecting them in a topological circle. A connected union of finitely many heteroclinic cycles is a heteroclinic network. These objects are associated with intermittent dynamics and used to model stop-and-go behaviour in various applications, including neuroscience, geophysics, game theory and populations dynamics. As examples we mention the work of Chossat and Krupa \cite{ChossatKrupa2014}, Rodrigues \cite{Rodrigues}, Aguiar \cite{Aguiar} and Hofbauer and Sigmund \cite{HofbauerSigmund}, respectively. A thourough theoretical understanding of their architecture, stability properties and dynamical effects is therefore of general interest. This study contributes to the latter by investigating a particular type of complicated dynamics in a neighbourhood of a network.

Complex behaviour near a heteroclinic network is often connected to the occurence of switching related to the network. There are different types of switching: at a node, along a connection, or infinite, leading to increasingly complex behaviour near the network. Switching at a node guarantees the existence of initial conditions, near an incoming connection to that node, whose trajectory follows any of the possible outgoing connections at the same node. This means the incoming path does not predetermine the outgoing choice at the node. Switching along a connection extends the notion of switching at a node to initial conditions whose trajectories follow a prescribed connection. Infinite switching ensures that any sequence of connections in the network is a possible path near the network.

The term switching has also been used to describe simpler dynamics where there is one change in the choices observed in trajectories. This is the case, for example, of Kirk and Silber \cite{KS}: The network is made of two cycles and trajectories are allowed to change from a neighbourhood of one cycle to a neighbourhood of the other cycle at most once. This change is referred to as switching, although this is a very mild instance of this phenomenon. In Castro et al.\ \cite{CLP2010} the expression railroad switching is used in relation to switching at a node.

Postlethwaite and Dawes \cite{PostlethwaiteDawes} find a form of complicated, though not infinite, switching leading to regular and irregular cycling near a network. There are several examples in the literature where the existence of infinite switching leads to chaotic behaviour near the network, see Aguiar et al.\ \cite{ACL2005}, and papers by Labouriau and Rodrigues \cite{LR2012, RL2014}. All the networks considered by these authors have at least one node at which the linearized vector field has complex eigenvalues.

Complex behaviour near a network can also arise from the presence of noise-induced switching, see Armbruster and coworkers \cite{ASK2003, StoneArmbruster} and Ashwin and Podvigina \cite{AP2010}. We do not address the presence of noise in this work.

To the best of our knowledge, there are no examples of vector fields, whose linearization at equilibria have no complex eigenvalues, supporting networks that exhibit infinite switching. On the contrary, Aguiar \cite{Aguiar} shows that if a heteroclinic network contains a subnetwork such as the one studied in \cite{KS}, then there is no switching along one of the connections of the network (along the common connection, for readers familiar with this network). The absence of switching along a connection precludes infinite switching and, therefore, chaotic behaviour near the network.

In this paper, we inquire into the possibility of various types of switching in networks made of simple cycles. We give sufficient conditions for the absence of switching along a connection, containing as a particular case the result in \cite{Aguiar}. We provide two examples illustrating the role of these sufficient conditions. Both examples exhibit some form of switching. For one of them we proceed to study some of the dynamic consequences of the presence of switching.

The paper is organized as follows: The next section consists of preliminary concepts and may be skipped by readers familiar with the subject. Section \ref{noswitch} contains the main results concerning switching, and Section \ref{section-bowtie} develops the dynamics for an example. Section \ref{conclusion} concludes.

\section{Preliminaries}\label{prelim}
We look at differential equations $\dot{x}=f(x)$, where $f$ is a vector field on $\R^n$, $n \geq 4$, which is $\Gamma$-equivariant for some finite group $\Gamma \subset O(n)$, that is,
$$
f(\gamma .x)=\gamma. f(x), \quad \forall \; \gamma \in \Gamma \; \; \forall \; x \in \R^n.
$$
A {\em heteroclinic cycle} is a set of finitely many equilibria $\xi_j$, $j=1, \hdots, m$, together with trajectories connecting them:
$$
[\xi_j \rightarrow \xi_{j+1}] \subset W^u(\xi_j) \cap W^s(\xi_{j+1}) \neq \emptyset \;\;\;\; (j=1,\hdots,m; \; \; \xi_{m+1}=\xi_1).
$$
Assuming that each connection $[\xi_j \rightarrow \xi_{j+1}]$ is of saddle-sink type in an invariant fixed-point subspace $P_j=\Fix(\Sigma_j)$, for some subgroup $\Sigma_j \subset \Gamma$, ensures robustness of the cycle to perturbations that respect the symmetry $\Gamma$.

In this paper we are interested mostly in heteroclinic networks consisting of \emph{simple} heteroclinic cycles. In fact, the properties of simple cycles are only required locally along a connection. There are various notions of simple cycles in the literature, see Krupa and Melbourne \cite{KrupaMelbourne2004} as well as Podvigina \cite{Podvigina2012, Podvigina2013}. The relevant properties for our results are
\begin{itemize}
 \item all $P_j$ are two-dimensional,
 \item the linearization $\mathrm{d}f(\xi_j)$ has only real and simple eigenvalues.
\end{itemize}
In particular, the cycles we address can be thought of as being generated by the simplex method used by Ashwin and Postlethwaite \cite{AP}, which we briefly describe in subsection \ref{simplex}. This guarantees that the equilibria $\xi_j$ are on the coordinate axes. We refer to networks constructed of such cycles as \emph{simple networks} and call the eigenvalues of the linearization $\mathrm{d}f(\xi_j)$ at the equilibria radial ($-r_j<0$), contracting ($-c_j<0$), expanding ($e_j>0$) and transverse ($t_j \gtrless 0$), see Krupa and Melbourne \cite[Section 2.3]{KrupaMelbourne95a}.

Simple cycles have been classified into different types. For our purposes it is sufficient to distinguish two, see \cite{Podvigina2012}. In $\R^4$, $n \geq 4$, a cycle is of 
\begin{itemize}
 \item \emph{type $A$} if the eigenspaces corresponding to $c_j, t_j, e_{j+1}$ and $t_{j+1}$ belong to the same $\Sigma_j$-isotypic component for all $j$, see \cite[page 1190 and Theorem 2.4(b)]{KrupaMelbourne2004};
 \item \emph{type $Z$} if $\Sigma_j$ decomposes $P_j^\perp$, the orthogonal complement of $P_j=\Fix(\Sigma_j)$, into one-dimensional $\Sigma_j$-isotypic components for all $j$, see \cite[Definition 8]{Podvigina2012}.
\end{itemize}
Note that the two types are mutually exclusive, because for an $A$ cycle the isotypic component containing $c_j, t_j, e_{j+1}$ and $t_{j+1}$ is at least two-dimensional. In $\R^4$, cycles of type $Z$ can be further divided into types $B$ and $C$, see \cite[Definition 3.2]{KrupaMelbourne2004}.

In order to establish a convenient set-up for our study we, as usual, define cross-sections to the flow at incoming and outgoing connections near an equilibrium. We use the notation $H_i^{\inn,j}$ to denote the cross-section at an incoming connection to $\xi_i$ from $\xi_j$. Similarly, $H_i^{\out,j}$ denotes the cross-section at an outgoing connection from $\xi_i$ to $\xi_j$. Local maps are defined using the linearized flow near each equilibrium $\xi_i$. We use the, by now, standard notation for local and global maps:
\begin{itemize}
	\item  $\phi_{jik}: \; H_i^{\inn,j} \rightarrow H_i^{\out,k}$ is the local map near $\xi_i$ describing the flow for points coming from $\xi_j$ and proceeding to $\xi_k$;
	\item  $\psi_{ij}: \; H_i^{\out,j} \rightarrow H_j^{\inn,i}$ is the global map along the connection $[\xi_i \rightarrow \xi_j]$. It maps radial and expanding directions near $\xi_i$ to contracting and radial directions near $\xi_j$. Depending on the symmetry of the system sometimes all $\psi_{ij}$ may be assumed to equal the identity. See \cite{KrupaMelbourne2004} for more on global maps.
\end{itemize}
For the purpose of defining these maps, it is often convenient to add a second index to the eigenvalues of $\mathrm{d}f(\xi_j)$. It corresponds to the other node in the respective connection, e.g.\ $-c_{jk}$ is the contracting eigenvalue at $\xi_j$ in the direction of $\xi_k$. In the context of networks, transverse eigenvalues to one cycle may be contracting or expanding with respect to another cycle, see Castro and Lohse \cite[Section 5]{CastroLohse2014} for an example.

\medbreak
We now recall several notions of switching from Aguiar et al.\ \cite{ACL2005}. A \emph{(finite) path} on a heteroclinic network $X$ is a (finite) sequence of connections in $X$ such that the end node of a connection coincides with the initial node of the following connection. A given path is \emph{followed} or \emph{shadowed}, if in every neighbourhood of $X$ there is a trajectory that follows the connections of the path in the prescribed order. We say that there is \emph{(finite) switching} near $X$ if every (finite) path is shadowed. A special case of finite switching, and a necessary condition for infinite switching, is \emph{switching along a trajectory} $[\xi_j \to \xi_k]$, where we only demand that every combination of incoming connections at $\xi_j$ and outgoing connections at $\xi_k$ is shadowed. The weakest form of switching, \emph{switching at a node}, occurs when points from any incoming trajectory can follow any outgoing trajectory at a node.

\subsection{A brief description of the simplex method}\label{simplex}

We refer the interested reader to the article of Ashwin and Postlethwaite \cite{AP}, which we follow here.

Given a directed graph with $n$ vertices, the {\em simplex method} provides a vector field for which the directed graph is realized as a heteroclinic network. In the noise free setting, a directed graph with $n$ vertices is realised for a vector field on $\R^n$ explicitly defined as
$$
\dot{x}_j = x_j\left(1-|x|^2+\sum_{i=1}^n a_{ij}x_i^2 \right),
$$
with suitably chosen coefficients $a_{ij}$. This vector field has $\Z_2^n$ symmetry with $\Z_2$ acting by reflection on each coordinate axis. Hence, each coordinate hyperplane is flow-invariant and the nodes of the heteroclinic network are placed one on each coordinate axis.

Proposition 1 of \cite{AP} proves that the simplex method realizes any directed graph which is one- and two-cycle free. This means that the graph neither has nodes which are the start and end point of the same edge (one-cycles), nor nodes $v_i$, $v_j$ ($i\neq j$) with edges $[v_i \rightarrow v_j]$ and $[v_j \rightarrow v_i]$ (two-cycles).

\section{Switching along common connections}\label{noswitch}
Consider a simple heteroclinic network in $\R^n, n \geq 4$, possibly (but not necessarily) generated by the simplex method in \cite{AP}, with all equilibria on the coordinate axes. We label them accordingly, i.e.\ $\xi_j$ lies on the $x_j$-axis.

In \cite[Theorem 1]{Aguiar} it is shown that for a heteroclinic network with a Kirk and Silber \cite{KS} subnetwork (see Figure \ref{kirksilber}) there is no switching along the common connection and therefore no infinite switching. The result is stated in the context of replicator dynamics\footnote{Replicator dynamics are described by equations of the form $\dot{x}_i=x_i(f_i(x)-\bar{f}(x))$, where $f_i$ represents the fitness of type/agent $i$ and $\bar{f}$ denotes the average fitness. In a game theoretical context $f_i(x)=\sum_j a_{ij}x_j$ where $a_{ij}$ represents, loosely speaking, the payoff obtained by a choice $i$ of strategies against a choice $j$. See Hofbauer and Sigmund \cite[chapter 7]{HofbauerSigmund} for a detailed explanation. Note that, {\em a priori}, there is no symmetry associated to these systems.} and the equations are naturally equipped with a $\Z_2^n$ symmetry, leading to the reasonable assumption that in the construction of return maps the global maps are the identity. Fixed point planes are generated by groups isomorphic to $\Z_2^{n-2}$ and the resulting cycles are of type $Z$. We generalize this result by proving the absence of switching along any connection for which the respective incoming and outgoing directions span the same space (in a sense that is made precise in Assumption 1 below).

\begin{figure}[!htb]
\centerline{\includegraphics[width=0.4\textwidth]{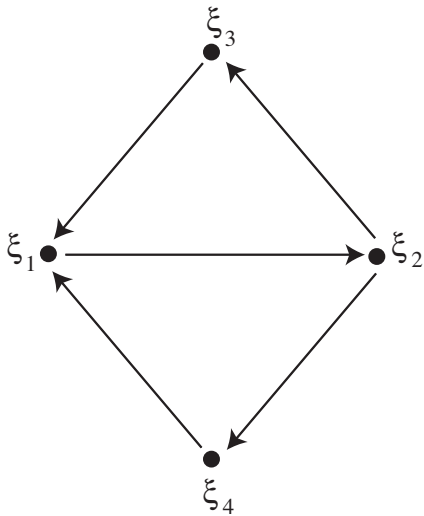}}
\caption{A Kirk and Silber network\label{kirksilber}}
\end{figure}

Suppose that in our heteroclinic network two cycles share a connection $[\xi_1 \to \xi_2]$. Let the first (second) cycle have an incoming direction to $\xi_1$ from $\xi_\alpha$ ($\xi_a$) and an outgoing from $\xi_2$ to $\xi_\beta$ ($\xi_b$), so that the following sequences of connections, shown in Figure \ref{commonconnection}, belong to the respective cycles:
\begin{align*}
[\xi_\alpha \to \xi_1 \to \xi_2 \to \xi_\beta] \quad \text{and} \quad
[\xi_a \to \xi_1 \to \xi_2 \to \xi_b]
\end{align*}
\begin{figure}[!htb]
\centerline{\includegraphics[width=0.6\textwidth]{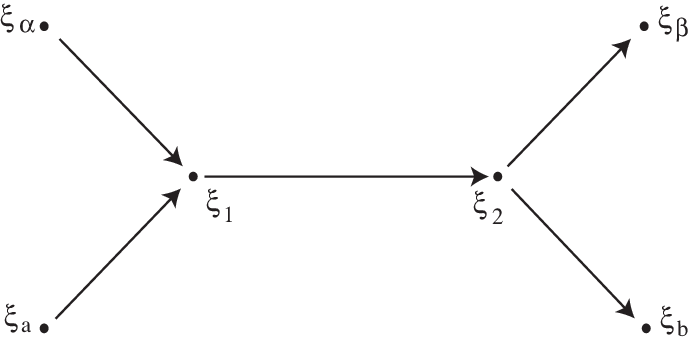}}
\caption{The common connection\label{commonconnection}}
\end{figure}
We assume $\xi_\alpha \neq \xi_a$ and $\xi_\beta \neq \xi_b$, but allow $\xi_\alpha=\xi_\beta$ or even $\xi_\alpha=\xi_2$ and $\xi_\beta=\xi_1$, and analogously for the other cycle. If $\xi_{\alpha}=\xi_2$ (and $\xi_{\beta}=\xi_1$), then the network contains the two-cycle $[\xi_1 \to \xi_2 \to \xi_1]$. Therefore, it cannot be obtained from the simplex method as pointed out in subsection \ref{simplex}. An example of such a network is the $(B_2^+,B_2^+)$ network studied in \cite{CastroLohse2014}, and the present results still hold for this and other networks with two-cycles.

Furthermore, suppose the following:

\paragraph{Assumption 1:} The $x_ax_\alpha$-plane is mapped into the $x_bx_\beta$-plane by the global map $\psi_{12}$ along the common connection.
\medbreak
This assumption is fulfilled by all of the networks in Castro and Lohse \cite[Theorem 3.7]{CastroLohse2015}. In order to describe the relevant sets succinctly, we state precisely what we mean by the word \emph{cusp}.
\begin{definition}
 Consider a straight line $g$ through the origin in $\R^2$ and a curve $\gamma$ that is tangent to $g$ at $(0,0)$ and contained entirely in one of the half planes generated by $g$. Then, by reflecting $\gamma$ along $g$, a neighbourhood of the origin is divided into two open sets (each with two connected components), as in Figure \ref{cusp}. We call the one that is asymptotically (for $\varepsilon \to 0$) larger in $B_\varepsilon(0)$ a \emph{thick cusp} and the other one a \emph{thin cusp tangent to $g$}.
\end{definition}
\begin{figure}[!htb]
\centerline{\includegraphics[width=0.3\textwidth]{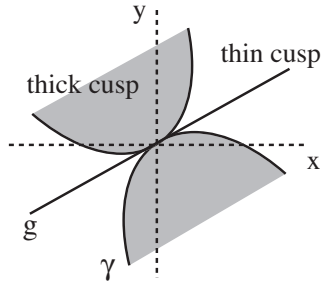}}
\caption{Thin and thick cusps\label{cusp}}
\end{figure}

Typically, cusps are given through curves of the form $y=\gamma(x)=x^\alpha$, the tangency axis then depending on $\alpha \gtrless 1$.

The following is a preliminary result.
\begin{lemma}\label{cusps}
 Let there be two curves defining cusps in $\R^2$ that are not tangent to the same line. Then in a sufficiently small neighbourhood of $(0,0)$
 \begin{itemize}
  \item each thick cusp contains the other thin cusp,
  \item the two thick cusps intersect in a set of positive measure,
  \item the two thin cusps do not intersect.
 \end{itemize}
\end{lemma}
\begin{proof}
 All three statements follow because the thin cusp is arbitrarily close to its tangency axis in $B_\varepsilon(0)$ for $\varepsilon>0$ sufficiently small, see Figure \ref{cusp-intersection}.
\end{proof}
\begin{figure}[!htb]
\centerline{\includegraphics[width=0.3\textwidth]{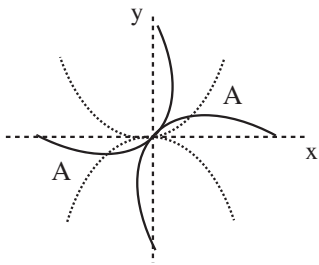}}
\caption{Intersecting cusps: region A is the intersection of thin cusps and occurs away from the origin.\label{cusp-intersection}}
\end{figure}

We now state and prove our main result.
\begin{theorem}\label{noswitching}
 In a simple heteroclinic network in $\R^n, n \geq 4$, with connections $[\xi_\alpha \to \xi_1 \to \xi_2 \to \xi_\beta]$ and $[\xi_a \to \xi_1 \to \xi_2 \to \xi_b]$ fulfilling Assumption 1, there is no switching along $[\xi_1 \to \xi_2]$.
\end{theorem}
\begin{proof}
 For both cycles we define $(n-1)$-dimensional cross sections transverse to the flow, and local maps (in local coordinates) between them in the standard way. We claim, and justify at the end of this proof, that only the two directions spanning the planes in Assumption 1 are relevant. Restricting to these and using the standard notation of $-c_{jk}$ and $e_{jl}$ for the contracting and expanding eigenvalues at $\xi_j$ we obtain:
\begin{align*}
&\phi_{\alpha 12}: \widehat{H}_1^{\inn, \alpha} \to \widehat{H}_1^{\out,2}, &&(x,y) \mapsto \left( x^\frac{c_{1\alpha}}{e_{12}},yx^\frac{c_{1a}}{e_{12}} \right) \\
&\phi_{a12}: \widehat{H}_1^{\inn, a} \to \widehat{H}_1^{\out,2}, &&(x,y) \mapsto \left(yx^\frac{c_{1\alpha}}{e_{12}},x^\frac{c_{1a}}{e_{12}}\right)\\
&\phi_{12\beta}: \widehat{H}_2^{\inn, 1} \to \widehat{H}_2^{\out,\beta}, &&(x,y) \mapsto \left(x^\frac{c_{21}}{e_{2\beta}},yx^{-\frac{e_{2b}}{e_{2\beta}}}\right)\\
&\phi_{12b}: \widehat{H}_2^{\inn, 1} \to \widehat{H}_2^{\out,b}, &&(x,y) \mapsto \left(y^\frac{c_{21}}{e_{2b}},xy^{-\frac{e_{2\beta}}{e_{2b}}}\right)
\end{align*}
Here $\widehat{H}$ denotes the restriction to the relevant plane in the cross section $H$. The maps are well-defined by Assumption 1. We have restricted further to $x,y \geq 0$, but the corresponding expressions for different signs can be derived directly from the maps above. The curve $y=\gamma_1(x)=x^\frac{c_{1a}}{c_{1\alpha}}$ is tangent to one of the coordinate axes at $(0,0)$ and divides $\widehat{H}_1^{\out,2}$ into two regions, separating points coming to $\xi_1$ from $\xi_\alpha$ from those coming from $\xi_a$. We denote these respectively by
\begin{align*}
C_{\alpha 1}&:= \left\{(x,y) \in \widehat{H}_1^{\out,2} \mid y<\gamma_1(x) \right\},\\
C_{a 1}&:= \left\{(x,y) \in \widehat{H}_1^{\out,2} \mid y>\gamma_1(x) \right\}.
\end{align*}
In the same way, the curve $y=\gamma_2(x)=x^\frac{e_{2b}}{e_{2\beta}}$ separates points going from $\xi_2$ to $\xi_\beta$ from those going to $\xi_b$ in $\widehat{H}_2^{\inn,1}$, giving rise to
\begin{align*}
E_{2 \beta}&:= \left\{(x,y) \in \widehat{H}_2^{\inn,1} \mid y<\gamma_2(x) \right\},\\
E_{2b}&:= \left\{(x,y) \in \widehat{H}_2^{\inn,1} \mid y>\gamma_2(x) \right\}.
\end{align*}
These sets are the domains of $\phi_{12b}$ and $\phi_{12\beta}$.

The global map $\psi_{12}: \widehat{H}_1^{\out,2} \to \widehat{H}_2^{\inn,1}$ is the same for both cycles and determines how points in a neighbourhood of the network follow the connection $[\xi_1 \to \xi_2]$. We investigate the intersection $\psi_{12}(\gamma_1) \cap \gamma_2$ to identify four regions of points in $\widehat{H}_2^{\inn,1}$, following the equilibria in the network in the order $\alpha12\beta$, $\alpha12b$, $a12\beta$ or $a12b$.

The $xy$-plane in $\widehat{H}_1^{\out,2}$ is split into two cusp-shaped regions by $\gamma_1$, which is tangent to the horizontal or vertical axis, depending on $c_{1a}\gtrless c_{1\alpha}$. Then $\psi_{12}(\gamma_1)$ is tangent to the image of this tangency axis.

The global map $\psi_{12}$ is an invertible linear map
\begin{equation*}
 \psi_{12}(x,y)=(\kappa_{11}x+\kappa_{12}y,\; \kappa_{21}x+\kappa_{22}y),
\end{equation*}
with $\kappa_{11}\kappa_{12}\kappa_{21}\kappa_{22} \neq 0$ generically. The coordinate axes in $\widehat{H}_1^{\out,2}$ are not necessarily fixed by $\psi_{12}$, so that the tangency axes in $\widehat{H}_2^{\inn,1}$ may be distinct. If that is the case, by Lemma \ref{cusps} the thin cusps do not intersect and so one of the paths is missing and there is no switching. If the tangency axes coincide, one of the cusps is always contained in the other, and again one of the paths is missing.
\medbreak
Finally, we justify our initial claim. Taking into account the remaining dimensions of the cross sections does not relax the contraints imposed on the $xy$-plane for points to be in the $C$- or $E$-sets. So the sets that do not intersect, ensuring the absence of switching, cannot intersect in a full neighbourhood of the origin.
\end{proof}

While we have presented the proof in a short version, more detailed information can be obtained about which paths are followed or not, when we distinguish the case where $\psi_{12}$ is the identity map and thus fixes the coordinate axes in $\widehat{H}_1^{\out,2}$. This is the case for type $Z$ cycles. Then the curve $\gamma_1$ is tangent to one of the axes and $\psi_{12}(\gamma_1)$ is tangent to the same axis in $\widehat{H}_2^{\inn,1}$. One of the cusps is contained in the other, depending on the relative size of the eigenvalues:
\begin{align*}
\frac{c_{1a}}{c_{1\alpha}}<\frac{e_{2b}}{e_{2\beta}} \quad \Leftrightarrow \quad \psi_{12}(C_{a 1}) \subset E_{2b} \quad \text{and} \quad \psi_{12}(C_{\alpha 1}) \supset E_{2 \beta}\\
\frac{c_{1a}}{c_{1\alpha}}>\frac{e_{2b}}{e_{2\beta}} \quad \Leftrightarrow \quad \psi_{12}(C_{a 1}) \supset E_{2b} \quad \text{and} \quad \psi_{12}(C_{\alpha 1}) \subset E_{2 \beta}
\end{align*}
So the sets $\psi_{12}(C_{a 1}) \cap E_{2b}$ and $\psi_{12}(C_{\alpha 1}) \cap E_{2 \beta}$ are non-empty. In the first instance $\psi_{12}(C_{a 1}) \subset E_{2b}$ prevents $\psi_{12}(C_{a1})$ from intersecting $E_{2 \beta}$. Analogously, in the second instance $\psi_{12}(C_{\alpha 1}) \cap E_{2b}= \emptyset$. This means there are always trajectories following the paths $a12b$ and $\alpha12\beta$, but only one of $a12\beta$ and $\alpha12b$. That is, there is one cycle from which it is impossible to switch to the other. This is in accordance with the findings for $(B_3^-, B_3^-)$ networks in \cite{KS}.

If $\psi_{12}$ is not the identity, as for type $A$ cycles, suppose first that $C_{a1}$ and $E_{2b}$ are thick cusps. Then $\psi_{12}(C_{a 1})$ is also thick and regardless of $\psi_{12}$ we have $\psi_{12}(C_{a 1}) \cap E_{2b} \neq \emptyset$, as well as $\psi_{12}(C_{\alpha 1}) \cap E_{2b} \neq \emptyset$ and $\psi_{12}(C_{a 1}) \cap E_{2 \beta} \neq \emptyset$, but $\psi_{12}(C_{\alpha 1}) \cap E_{2 \beta} = \emptyset$. So $a12b$, $\alpha12b$ and $a12\beta$ are realized, but not $\alpha12\beta$. 
In the same way, if $C_{a1}$ and $E_{2b}$ are thin cusps, all paths except $a12b$ are realized, which corresponds to $\psi_{12}(C_{a 1}) \cap E_{2b} = \emptyset$. Finally, for a thin cusp $C_{a1}$ and a thick cusp $E_{2b}$, the former generically lies in the latter. So all paths except for $a12\beta$ are realized. If it is the other way around, then $\alpha12b$ is not realized.

In contrast to type $Z$ cycles, for type $A$ cycles the missing path can be any of the four possible ones, depending on the relative size of the eigenvalues at $\xi_1$ and $\xi_2$. This means that for $A$ networks there are configurations in which it is possible to switch from either cycle to the other. But the price to pay for this is that for one cycle \emph{all} trajectories switch to the other one (and possibly return later).

\begin{remark}\label{remark}
 In the proof of Theorem \ref{noswitching} we use the fact that the connections depicted in Figure \ref{commonconnection} have the properties of connections in simple cycles, but make no assumption concerning the remaining connections of either cycle. Thus, Theorem \ref{noswitching} applies to networks other than simple.
\end{remark}

The absence of switching still holds if there are multiple common connections in a row, as long as Assumption 1 is fulfilled at the first and last common node.
\begin{corollary}
 If in Theorem \ref{noswitching} the connection $[\xi_1 \to \xi_2]$ is replaced by a sequence of connections $[\xi_1 \to \xi_1^* \to \xi_2^* \to \ldots \to \xi_k^* \to \xi_2]$, $k<\infty$, then there is no switching along this sequence.
\end{corollary}

Moreover, note that the sign of transverse eigenvalues is irrelevant. In particular, if there are more incoming connections at $\xi_1$ or more outgoing connections at $\xi_2$, the result still holds, because this only increases the number of sequences that would have to be shadowed in order to have switching along $[\xi_1 \to \xi_2]$.

We have thus shown that the key factors inhibiting switching are the presence of a common connection and the fact that the respective incoming and outgoing directions span the same space. The next two subsections show that, on the one hand, this geometry is essential, and, on the other hand, that it is satisfied for a large class of cycles.

\subsection{Switching in the House network}
The \emph{House network} has five equilibria, one on each coordinate axis in $\R^5$. It can be generated with the simplex method from \cite{AP} and consists of a $B_3^-$ and a $C_4^-$ cycle (both of type $Z$), see Figure \ref{house}. There is a common connection, but Assumption 1 does not hold.

\begin{figure}[!htb]
\centerline{\includegraphics[width=0.4\textwidth]{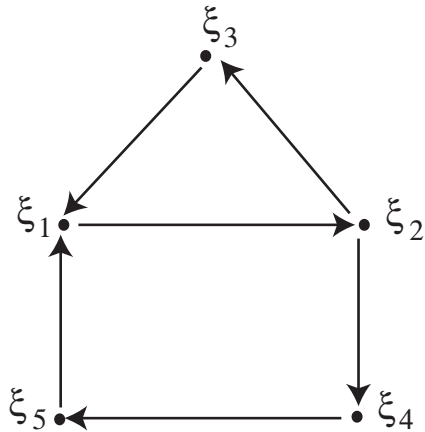}}
\caption{The House network\label{house}}
\end{figure}

\begin{lemma}\label{switch-house}
 There is switching along the common connection in the House network.
\end{lemma}
\begin{proof}
 Let the $B_3^-$ cycle be $[\xi_1 \to \xi_2 \to \xi_3]$ and the $C_4^-$ cycle $[\xi_1 \to \xi_2 \to \xi_4 \to \xi_5]$, so that the common connection is $[\xi_1 \to \xi_2]$. The $x_ax_\alpha$-plane in Assumption 1 is here the $x_3x_5$-plane, whereas the $x_bx_\beta$-plane is the $x_3x_4$-plane. Since the cycles are of type $Z$, the global map $\psi_{12}$ is the identity and hence Assumption 1 does not hold.
 
 We label the cross sections $H_j^{\inn,i}$ and $H_j^{\out,k}$ just as before. They are now four-dimensional and there are three relevant directions (ignoring the radial one as usual), so we can think of them as cubes. Then with the same reasoning as above, $H_1^{\out,2}$ is split into the set of points coming from $\xi_3$ and from $\xi_5$ by a surface given through a curve in the $x_3x_5$-plane that extends trivially in the $x_4$-direction, see Figure \ref{H1-1}, which shows an octant of a neighbourhood of the origin in $H_1^{\out,2}$.
 \begin{figure}[!htb]
\centerline{\includegraphics[width=0.4\textwidth]{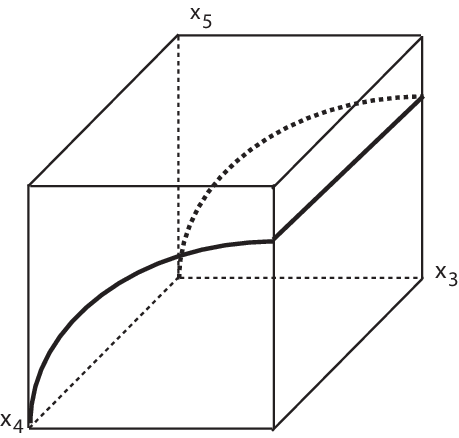}}
\caption{Splitting of $H_1^{\out,2}$ into the sets of points coming from $\xi_3$ and from $\xi_5$. The shape of the dividing surface depends on the relative magnitudes of the eigenvalues $-c_{13}$ and $-c_{15}$ at $\xi_1$.\label{H1-1}}
\end{figure}
 
 Similarly, $H_2^{\inn,1}$ is split into points going to $\xi_3$ and $\xi_4$, respectively, by a surface given through a curve in the $x_3x_4$-plane extending trivially in the $x_5$-direction.
 
 Since the cycles are of type $Z$, we may assume all global maps to be the identity. Then there are four regions of intersections in $H_2^{\inn,1}$, corresponding to all four combinations of coming from $\xi_3$ or $\xi_5$ and going to $\xi_3$ or $\xi_4$, see Figure \ref{H1-2}:
 \begin{itemize}
  \item from $\xi_3$ to $\xi_3$: sufficiently close to the $x_3$-axis
  \item from $\xi_3$ to $\xi_4$: sufficiently close to the $x_4$-axis
  \item from $\xi_5$ to $\xi_3$: sufficiently close to the $x_5$-axis
  \item from $\xi_5$ to $\xi_4$: sufficiently close to the $x_4x_5$-plane, but away from the $x_5$-axis
 \end{itemize}
  \begin{figure}[!htb]
\centerline{\includegraphics[width=0.4\textwidth]{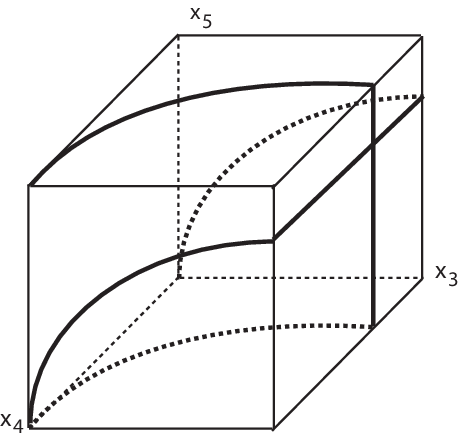}}
\caption{Intersection of the splitting into points coming from $\xi_3$ or $\xi_5$ and going to $\xi_3$ or $\xi_4$ in $H_2^{\inn,1}$.\label{H1-2}}
\end{figure}
 This shows that all possible sequences are shadowed by trajectories and thus we have switching along $[\xi_1 \to \xi_2]$. Note that our reasoning does not depend on the cusps being thin or thick.
\end{proof}

\subsection{Generality of Assumption 1}
We now illustrate that Assumption 1 holds for a large class of networks. Note that the global maps $\psi_{ij}$ map $\Sigma_i$-isotypic components into themselves. Therefore, if all cycles in the network are of type $Z$, they preserve coordinate axes, because all isotypic components except for $\Fix(\Sigma_i)$ are one-dimensional. So, for type $Z$ cycles, Assumption 1 requires that the $x_ax_\alpha$-plane equals the $x_bx_\beta$-plane, which essentially means that the $x_\alpha$-axis is the same as the $x_\beta$-axis and similarly for the $x_a$- and $x_b$-axes. In the strict context of simple networks, this applies to all type $Z$ cycles with a Kirk and Silber subnetwork, similar to \cite[Theorem 1]{Aguiar}. However, as noted in Remark \ref{remark}, our result still holds in a wider context. In particular, it also applies to networks with equilibria $\xi_\alpha \neq \xi_\beta$, but both on the same coordinate axis, which is not allowed in simple networks.

If a cycle is not of type $Z$, Assumption 1 still holds as long as the relevant coordinate subspaces are preserved. This allows for isotypic decompositions into arbitrary components with the only restriction being that there is a two-dimensional $\Sigma_1$-isotypic component containing the eigenspaces corresponding to $c_{a1}$ and $ c_{\alpha 1}$ that is mapped by $\Psi_{12}$ into the $\Sigma_2$-isotypic component containing those corresponding to $e_{2b}$ and $e_{2 \beta}$. This is satisfied for a number of cycles of type $A$ and does not require a Kirk and Silber subnetwork.

Furthermore, in $\R^4$ Assumption 1 holds for all simple networks listed in \cite{CastroLohse2014}.

\medbreak
We now briefly comment on the necessity of Assumption 1 for the absence of switching. For this, we look at networks with a connection structure as described above, and restrict to the case where the global maps equal the identity. 

Looking at the local maps it is easy to see that points in $H_1^{\out,2}$ coming from $\xi_i$ are located near a hyperplane not containing the $x_j$-axis, where $j \neq i$ and $i,j$ are either $\alpha$ or $a$. Also, points in $H_2^{\inn,1}$ going to $\xi_k$ are located near a hyperplane not containing the $x_l$-axis, where $l \neq k$ and $k,l$ are either $\beta$ or $b$. If Assumption 1 does not hold, the $x_ax_\alpha$-plane is not taken to the $x_bx_\beta$-plane and so, at most one of these axes is common to the restricted cross sections $\widehat{H}_1^{\out,2}$ and $\widehat{H}_2^{\inn,1}$. If exactly one axis is common, the proof for switching along $[\xi_1 \to \xi_2]$ follows as for the House network. If no axis is common then points go from $\xi_i$ to $\xi_k$ if they are near the $x_ix_k$-plane.

This shows that, under these circumstances, breaking Assumption 1 creates switching along the connection. We do not address the case of general global maps.

\section{Dynamics near the Bowtie network}\label{section-bowtie}
The purpose of this section is to provide an illustration of some types of behaviour that can be observed near a simple network in $\R^5$ that does not satisfy Assumption 1.

The Bowtie network consists of two cycles, each with three equilibria, connected at an equilibrium as depicted in Figure \ref{bowtie}. The lowest dimension for this construction is $\R^5$. Considering the cycles' position in the network, we name them the $L$-cycle and the $R$-cycle. With the chosen labelling of the equilibria, the $R$-cycle is exactly the $\xi_3$-cycle in Kirk and Silber \cite{KS} and Castro and Lohse \cite{CastroLohse2014}, with an extra transverse dimension.

Recall that putting the $R$- and $L$-cycles together in a network in $\R^4$ by means of a common connecting trajectory, and ensuring that all transverse eigenvalues are negative whenever possible, leads to the dynamics studied by Kirk and Silber \cite{KS}. The asymptotic behaviour of points near the network is restricted to going around one of the two cycles and an orbit starting near one of the cycles and switching to the other one will never switch back. We show that by reducing the common elements between the cycles to a node (adding an extra dimension is necessary to achieve this), we observe much richer dynamics near the network.

\begin{figure}[!htb]
\centerline{\includegraphics[width=0.35\textwidth]{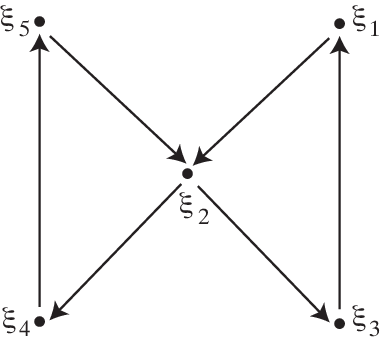}}
\caption{The Bowtie network\label{bowtie}}
\end{figure}

In order to describe the dynamics near the network, we use the notation for local and global maps from Section \ref{prelim}, all of which are listed in Appendix A. The $\Z_2^n$ symmetry allows us to consider the global maps $\psi_{ij}$ as equal to the identity and thus we freely identify $H_i^{\out,j}$ with $H_{j}^{\inn,i}$.

We further denote the return map around the $R$-cycle as
$h_R: \; H_1^{\out,2} \rightarrow H_1^{\out,2}$ given by $h_R=\phi_{312} \circ \phi_{231} \circ \phi_{123}$. Analogously, the return map around the $L$-cycle is $h_L: \; H_5^{\out,2} \rightarrow H_5^{\out,2}$ given by $h_L= \phi_{452} \circ \phi_{245} \circ \phi_{524}$.

In order to study the transitions from one cycle to the other, we define $g_{RL}: H_1^{\out,2} \rightarrow H_5^{\out,2}$, taking points from a neighbourhood of the $R$-cycle to a neighbourhood of the $L$-cycle, by $g_{RL}=\phi_{452} \circ \phi_{245} \circ \phi_{124}$. In the opposite direction, we define $g_{LR}: H_5^{\out,2} \rightarrow H_1^{\out,2}$ by $g_{LR}=\phi_{312} \circ \phi_{231} \circ \phi_{523}$.

\begin{figure}[!htb]
\centerline{\includegraphics[width=0.5\textwidth]{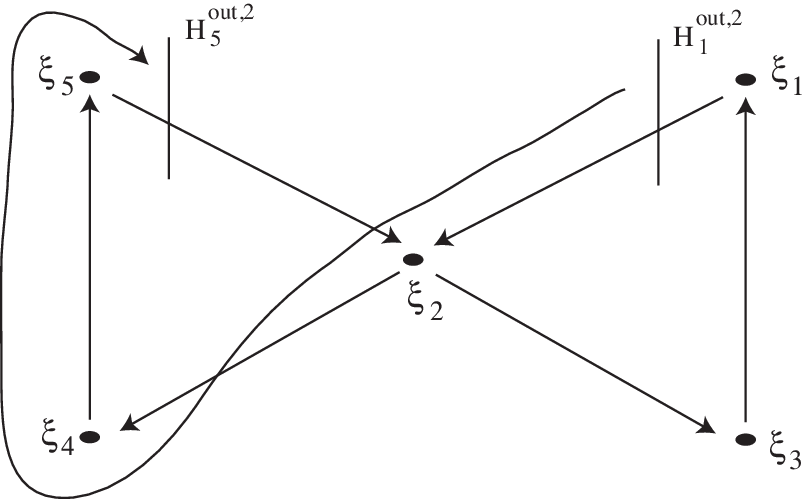}}
\caption{The map $g_{RL}$.\label{gRL}}
\end{figure}

The following parameters, dependent on the eigenvalues at the nodes, will be of use for our study. We use the symbol $\tilde{\mbox{}} $ for objects concerning the $R$-cycle:
$$
\begin{array}{lcl}
\rho=\dfrac{c_{42}c_{54}c_{25}}{e_{24}e_{45}e_{52}} & \mbox{\hspace{2cm}} & \tilde{\rho} = \dfrac{c_{32}c_{13}c_{21}}{e_{23}e_{31}e_{12}} \\
& & \\
\nu=-\dfrac{e_{23}}{e_{24}}+\dfrac{c_{25}c_{43}}{e_{24}e_{45}}+\dfrac{c_{53}c_{42}c_{25}}{e_{45}e_{24}e_{52}}& \mbox{\hspace{2cm}} & \tilde{\nu}=-\dfrac{e_{24}}{e_{23}}+\dfrac{c_{21}c_{34}}{e_{23}e_{31}}+\dfrac{c_{14}c_{32}c_{21}}{e_{31}e_{23}e_{12}} \\
& & \\
\mu=\dfrac{c_{21}}{e_{24}}+\dfrac{c_{25}c_{41}}{e_{24}e_{45}}+\dfrac{c_{51}c_{42}c_{25}}{e_{45}e_{24}e_{52}}& \mbox{\hspace{2cm}} & \tilde{\mu}=\dfrac{c_{25}}{e_{23}}+\dfrac{c_{21}c_{35}}{e_{23}e_{31}}+\dfrac{c_{15}c_{32}c_{21}}{e_{31}e_{23}e_{12}} \\
& & \\

\delta = \dfrac{c_{43}}{e_{45}}+\dfrac{c_{53}c_{42}}{e_{52}e_{45}}-\dfrac{e_{23}c_{54}c_{42}}{e_{52}e_{45}e_{24}} & \mbox{\hspace{2cm}} & \tilde{\delta} = \dfrac{c_{34}}{e_{31}}+\dfrac{c_{14}c_{32}}{e_{12}e_{31}}-\dfrac{e_{24}c_{13}c_{32}}{e_{12}e_{31}e_{23}} \\
\end{array}
$$
Note that even though we use some of the letters in \cite{KS}, the parameters pertaining to the $L$-cycle depend on eigenvalues different from those in \cite{KS}. Our choice of label for the nodes of the network does, however, ensure that the parameters pertaining to the $R$-cycle coincide with those pertaining to the $\xi_3$-cycle in \cite{KS}. The main point for those readers familiar with \cite{KS} is that for the Bowtie network the sign of $\delta $ is independent of that of $\tilde{\delta}$.
The following relations are also of use:
$$
\frac{e_{24}}{e_{23}}(\tilde{\rho}-1)-\tilde{\nu} = -\frac{c_{21}}{e_{23}}\tilde{\delta} \; \; \mbox{and  } \; \; \; \frac{e_{23}}{e_{24}}(\rho-1)-\nu = -\frac{c_{25}}{e_{24}}\delta.
$$

In our study we make use of the following convention:
\begin{enumerate}
	\item[(i)]  A point {\em visits} the $R$-cycle if its trajectory follows the sequence $[\xi_3 \rightarrow \xi_1]$. Analogously, a point {\em visits} the $L$-cycle if its trajectory follows the sequence $[\xi_4 \rightarrow \xi_5]$.
	\item[(ii)]  A point {\em takes a turn} around the $R$-cycle if its trajectory follows the sequence $[\xi_1 \rightarrow \xi_2 \rightarrow \xi_3]$. Analogously, a point {\em takes a turn} around the $L$-cycle if its trajectory follows the sequence $[\xi_5 \rightarrow \xi_2 \rightarrow \xi_4]$.
\end{enumerate}

We use a sequence of the letters $L$ and $R$ to describe the behaviour of a trajectory in a neighbourhood of the network, writing an $L$ for each visit to the left cycle and an $R$ for each visit to the right. For example, the sequence $RLR$ represents a trajectory visiting the right cycle, then the left and then the right again.

Note that there is switching at the node $\xi_2$ in the Bowtie network, i.e.\ there are initial conditions following the sequences $LL$, $RR$, $LR$ and $RL$. This is clear because the four local maps around $\xi_2$ all have domains of positive measure. For the Bowtie network a new definition of switching applies, namely \emph{switching along a cycle}. This is the natural extension of switching along a connection to a sequence of connections forming a heteroclinic cycle. The main point of this section is to prove switching along the cycle for the Bowtie network and point out some of its consequences.

\begin{lemma}
 There is switching along each cycle in the Bowtie network.
\end{lemma}
\begin{proof}
  As usual the cross section $H_2^{\out,3}$ is split into two sets of points coming from $\xi_1$ and points coming from $\xi_5$ (in the $x_1x_5$-plane). The cross section $H_2^{\inn,1}$ is split into points going to $\xi_3$ and $\xi_4$ (in the $x_3x_4$-plane). We have to check what happens to the sets in $H_2^{\out,3}$ when they move around the cycle until they hit $H_2^{\inn,1}$. This is described by the composition of local maps $\phi_{312} \circ \phi_{231}$.
 
 In the three relevant components we have
 \begin{equation*}
\phi_{231}(x_1,x_4,x_5)=\left(x_1^{\frac{c_{32}}{e_{31}}},x_4x_1^{\frac{c_{34}}{e_{31}}},x_5x_1^{\frac{c_{35}}{e_{31}}}\right) \cong (x_2,x_4,x_5),
\end{equation*}
and because
 \begin{equation*}
  \phi_{231}(x_1,0,x_1^\alpha)=\left(x_1^{\frac{c_{32}}{e_{31}}},0,x_1^{\alpha+\frac{c_{35}}{e_{31}}}\right),
 \end{equation*}
 we see that a cusp in the $x_1x_5$-plane becomes a cusp in the $x_2x_5$-plane under $\phi_{231}$, possibly changing its tangency axis. This means that the cusp in the $x_1x_5$-plane in $H_2^{\out,3}$ gets mapped to a cusp in the $x_3x_5$-plane in $H_2^{\inn,1}$. The intersection with the cusp given in the $x_3x_4$-plane can be studied as in the proof of Lemma \ref{switch-house}, and we have switching along the cycle.
\end{proof}

In order to be able to describe the number of consecutive turns around a cycle that points take, we define 
\begin{itemize}
	\item  the set of points that take at least one turn around the $R$-cycle
	$$
	\tilde{\E}_1 = \{ (1,x_2,x_3,x_4,x_5) \in H_2^{\inn,1}: \; \; (1,x_2,x_3,x_4,x_5) \in D_{h_R} \};
	$$
	We have $D_{h_R} = \left\{ (1,x_2,x_3,x_4,x_5) \in H_1^{\out,2}: \; x_4 < x_3^{\frac{e_{24}}{e_{23}}} \right\}$.
	\item  the set of points that take at least one turn around the $L$-cycle
	$$
	\E_1 = \{ (x_1,x_2,x_3,x_4,1) \in H_2^{\inn,5}: \; \; (x_1,x_2,x_3,x_4,1) \in D_{h_L} \}.
	$$
	Note that $D_{h_L} = \left\{ (x_1,x_2,x_3,x_4,1) \in H_5^{\out,2} : \; x_4 > x_3^{\frac{e_{24}}{e_{23}}} \right\}$.
\end{itemize}
We can extend these definitions (see Appendix A for the maps $h_.^n$) to describe the sets of points that take at least $n \geq 2$ turns around a given cycle as
\begin{eqnarray*}
\tilde{\E}_n & = & \{ (1,x_2,x_3,x_4,x_5) \in H_2^{\inn,1}: \; \; h_R^{n-1}(1,x_2,x_3,x_4,x_5) \in  D_{h_R} \}, \\
\E_n & = & \{ (x_1,x_2,x_3,x_4,1) \in H_2^{\inn,5}: \; \; h_L^{n-1}(x_1,x_2,x_3,x_4,1) \in  D_{h_L} \}.
\end{eqnarray*}
It is clear from the definition that $\E_{n+1} \subseteq \E_n$ for all $n \in \N$. Analogously for $\tilde{\E}_n$. As in \cite{KS} we assume for concreteness that $e_{23}>e_{24}$. Note that this means the $L$-cycle cannot attract a set of initial conditions that is asymptotically of full measure\footnote{Equivalently, the $L$-cycle cannot be \emph{essentially asymptotically stable}, see Melbourne \cite[Definition 1.1]{Melbourne1991}, and Lohse \cite{Lohse2015} for a relation to the geometry of the basin of attraction.}, because $D_{h_L}$ is a small cusp-shaped region.

In the next results we establish when the dynamics near the Bowtie network consists in points taking an infinite number of turns around the cycle near which they start. This is the situation observed for the Kirk and Silber network and the one we are not interested in.

\begin{lemma}\label{KS_behaviour}
If $\delta >0$ then $\E_{1} = \E_2$. 
If $\tilde{\delta} >0$ then $\tilde{\E}_{1} = \tilde{\E}_2$.
\end{lemma}

\begin{proof}
We present the proof for the $R$-cycle by establishing when $\tilde{\E}_{1} \subseteq \tilde{\E}_2$. This occurs when
the image by $h_R$ of any point in $\tilde{\E}_{1}$ belongs to $\tilde{\E}_{1}$. Using the maps in Appendix A and selecting only the relevant coordinates we have
$$
h_R(x_3,x_4,x_5) = (x_3^{\tilde{\rho}}, x_4 x_3^{\tilde{\nu}}, x_5 x_3^{\tilde{\mu}})
$$
and $D_{h_R} = \{ (x_3,x_4,x_5) \in H_1^{\out,2} : \; x_4 < x_3^{e_{24}/e_{23}} \}$. Then $h_R(x_3,x_4,x_5) \in \tilde{\E}_{1}$ is equivalent to 
$$
x_4 x_3^{\tilde{\nu}} < x_3^{\tilde{\rho}e_{24}/e_{23}} \Leftrightarrow x_4 < x_3^{\tilde{\rho}e_{24}/e_{23}-\tilde{\nu}}.
$$
Since $(x_3,x_4,x_5) \in \tilde{\E}_{1}$ to begin with, we have
$$
x_3^{\tilde{\rho}e_{24}/e_{23}-\tilde{\nu}} > x_3^{e_{24}/e_{23}} \Leftrightarrow \tilde{\rho}\frac{e_{24}}{e_{23}}-\tilde{\nu} < \frac{e_{24}}{e_{23}} \Leftrightarrow -\frac{c_{21}}{e_{23}}\tilde{\delta} <0  \Leftrightarrow \tilde{\delta} >0.
$$
\end{proof}

When $\delta, \tilde{\delta}<0$, the next proposition rules out the possibility of the trajectory of an initial condition near a given cycle remaining for ever near that same cycle. In other words, all points eventually switch from one cycle to the other. This rules out sequences of visits to the cycles of the form $RRR\hdots$ and $LLL\hdots$ as well. It is then immediate that infinite switching as defined by Aguiar {\em et al.} \cite{ACL2005} does not occur near the Bowtie network.

\begin{proposition}
If $\delta, \tilde{\delta} <0$ then $\tilde{\E}_{n+1} \neq \tilde{\E}_n$ and $\E_{n+1} \neq \E_n$ for all $n \in \N$. Furthermore, for all $x \in \tilde{\E}_1$ there is $n \in \N$ such that $x \notin \tilde{\E}_n$, and analogously for $x \in \E_1$.
\end{proposition}

\begin{proof}
Again we present the proof for the $R$-cycle. We show that $\tilde{\E}_{n+1} \varsubsetneq \tilde{\E}_n$. 

We have 
$$
h_R^n(x_3,x_4,x_5) = (x_3^{\tilde{\rho}^n}, x_4 x_3^{\tilde{\nu}\sum_{i=0}^{n-1}\tilde{\rho}^i}, x_5 x_3^{\tilde{\mu}\sum_{i=0}^{n-1}\tilde{\rho}^i}).
$$
According to the definition 
\begin{eqnarray*}
(x_3,x_4,x_5) \in \tilde{\E}_{n+1} & \Leftrightarrow & h_R^n(x_3,x_4,x_5) \in \tilde{\E}_1\\
& \Leftrightarrow & x_4 x_3^{\tilde{\nu}\sum_{i=0}^{n-1}\tilde{\rho}^i} < x_3^{\frac{e_{24}}{e_{23}}\tilde{\rho}^n} \\
& \Leftrightarrow & x_4 < x_3^{\frac{e_{24}}{e_{23}}\tilde{\rho}^n-\tilde{\nu}\sum_{i=0}^{n-1}\tilde{\rho}^i}
\end{eqnarray*}
Since $\tilde{\delta} <0$, the sequence $\left( \frac{e_{24}}{e_{23}}\tilde{\rho}^n-\tilde{\nu}\sum_{i=0}^{n-1}\tilde{\rho}^i \right)_{n \geq 1}$ is increasing. In fact,
\begin{eqnarray*}
 \frac{e_{24}}{e_{23}}\tilde{\rho}^{n+1}-\tilde{\nu}\sum_{i=0}^{n}\tilde{\rho}^i - \frac{e_{24}}{e_{23}}\tilde{\rho}^n+\tilde{\nu}\sum_{i=0}^{n-1}\tilde{\rho}^i 
& = & \frac{e_{24}}{e_{23}}\tilde{\rho}^n (\tilde{\rho}-1) - \tilde{\nu}\tilde{\rho}^n \\
& = & -\frac{c_{21}}{e_{23}}\tilde{\delta} \tilde{\rho}^n > 0 \\
& \Leftrightarrow & \tilde{\delta} < 0.
\end{eqnarray*}
From this it also follows that the difference between two consecutive terms is increasing and therefore the sequence is unbounded, which proves the second part of the claim. Hence, there exists $n \in \N$, such that $(x_3,x_4,x_5) \notin \tilde{\E}_n$, meaning that such a point takes at most $n-1$ turns around the $R$-cycle.
\end{proof}

This result is in accordance with the instability conditions in \cite[Lemma 4.4 (c)]{Lohse2015} for $B_3^-$-cycles in $\R^4$. It leads to the next theorem and from now on we assume its hypotheses are satisfied.

\begin{theorem}
If $\delta, \tilde{\delta} <0$ then neither cycle is e.a.s. In particular, there is no initial condition whose trajectory is asymptotic to just one cycle.
\end{theorem}

Notice that, when $\delta, \tilde{\delta} <0$, given any natural number we can find a point whose trajectory turns exactly that number of times around either of the cycles.

Having established that every point near a cycle will eventually be taken close to the other cycle, we proceed to give a better understanding of how this transition occurs. We start with a point near the $R$-cycle whose trajectory visits the $L$-cycle, that is, a point $(x_3,x_4,x_5) \in H_1^{\out,2} \backslash \tilde{\E}_1$. The transition in the opposite direction is in every way analogous. The map describing this transition is $g_{RL}: H_1^{\out,2} \rightarrow H_5^{\out,2}$ given by
$$
g_{RL}(x_3,x_4,x_5) = \left(x_4^{\mu}x_5^{\alpha}, x_3 x_4^{\nu} x_5^{\beta}, x_4^{\rho} x_5^{\frac{e_{24}}{c_{25}}\rho}\right)
$$
with 
$$
\alpha = \frac{c_{41}}{e_{45}}+\frac{c_{42}c_{51}}{e_{45}e_{52}}>0 \; \; \mbox{ and  } \; \; \; \beta = \frac{c_{43}}{e_{45}}+\frac{c_{42}c_{53}}{e_{45}e_{52}} = \delta+\frac{e_{23}}{c_{25}}\rho >0.
$$
The next result establishes that there exist both points $(x_3,x_4,x_5) \in H_1^{\out,2} \backslash \tilde{\E}_1$ that visit the $L$-cycle and immediately return to a neighbourhood of the $R$-cycle and points $(x_3,x_4,x_5) \in H_1^{\out,2} \backslash \tilde{\E}_1$ that visit the $L$-cycle and remain in its neighbourhood for at least one turn. These correspond, respectively, to dynamics containing the sequence $RLR$ and the sequence $RLL$.

\begin{proposition}\label{RLR_LRL-transition}
A point can be chosen in $H_1^{\out,2} \backslash \tilde{\E}_1$ so that either $g_{RL}(x_3,x_4,x_5) \in H_5^{\out,2} \backslash \E_1$ or $g_{RL}(x_3,x_4,x_5) \in \E_1$.
\end{proposition}

\begin{proof}
Since $(x_3,x_4,x_5) \in H_1^{\out,2} \backslash \tilde{\E}_1$ we know that $x_4 > x_3^{\frac{e_{24}}{e_{23}}}$. We have $g_{RL}(x_3,x_4,x_5) \in H_5^{\out,2} \backslash \E_1$ if and only if\footnote{Recall that, using only relevant coordinates, we have $g_{RL}(x_3,x_4,x_5) = (X_1,X_3,X_4)$.}
\begin{equation}\label{RLR-transition}
x_4^{\rho} x_5^{\frac{e_{24}}{c_{25}}\rho} < \left( x_3 x_4^{\nu} x_5^{\beta} \right)^{\frac{e_{24}}{e_{23}}}  \Leftrightarrow  x_4^{\frac{e_{23}}{e_{24}}\rho-\nu} x_5^{\frac{e_{23}}{c_{25}}\rho-\beta} < x_3.
\end{equation}
Note that 
\begin{itemize}
	\item  $\frac{e_{23}}{e_{24}}\rho-\nu = \frac{e_{23}}{e_{24}} - \frac{c_{25}}{e_{24}}\delta > \frac{e_{23}}{e_{24}} >1$ and
	\item  $\frac{e_{23}}{c_{25}}\rho-\beta = -\delta >0$
\end{itemize}
so that if we choose $x_5 \ll 1$ the inequality holds. We have thus found a point whose trajectory starts near the $R$-cycle, visits the $L$-cycle and immediately visits the $R$-cycle again.

For  finding a point corresponding to a sequence $RLL$, the opposite inequality in (\ref{RLR-transition}) needs to be satisfied. This can be achieved by choosing $x_3 \ll 1$.
\end{proof}

We now consider points $(x_3,x_4,x_5) \in H_1^{\out,2} \backslash \tilde{\E}_1$ that visit the $L$-cycle and remain in its neighbourhood for at least one turn. We show that given a natural number $n$, it is possible to find a point in $H_1^{\out,2} \backslash \tilde{\E}_1$ that takes at least $n$ turns around the $L$-cycle before visiting the $R$-cycle again. The dynamics exhibited by these points is described by transitions containing a sequence of the form $RLL\hdots LR$ where the $L$'s can appear in any finite number.

\begin{proposition}\label{finiteL-transition}
Points $(x_3,x_4,x_5) \in H_1^{\out,2} \backslash \tilde{\E}_1$ can take any finite number of turns around the $L$-cycle before returning to a neighbourhood of the $R$-cycle.
\end{proposition}

\begin{proof}
We show that it is possible to choose coordinates $(x_3,x_4,x_5)$ such that $g_{RL}(x_3,x_4,x_5) \in \E_{n+1}$. In the relevant coordinates, recall that  $g_{RL}(x_3,x_4,x_5) = (X_1,X_3,X_4)$. Points in $\E_{n+1}$ satisfy 
$$
X_3 < X_4^{\frac{e_{23}}{e_{24}}\rho^n-\nu \sum_{i=0}^{n-1}\rho^i}.
$$
Hence,
\begin{eqnarray*}
g_{RL}(x_3,x_4,x_5) \in \E_{n+1} & \Leftrightarrow & x_3 x_4^{\nu} x_5^{\beta} < \left( x_4^{\rho} x_5^{\frac{e_{24}}{c_{25}}\rho} \right)^{\frac{e_{23}}{e_{24}}\rho^n-\nu \sum_{i=0}^{n-1}\rho^i} \\
& \Leftrightarrow & x_3 < x_4^{\frac{e_{23}}{e_{24}}\rho^{n+1}-\nu \sum_{i=0}^{n}\rho^i -\nu} x_5^{\frac{e_{24}}{c_{25}}\left(\frac{e_{23}}{e_{24}}\rho^{n+1}-\nu \sum_{i=0}^{n}\rho^i\right) -\beta}
\end{eqnarray*}
Since the sequence $\left( \frac{e_{24}}{e_{23}}\rho^n-\nu\sum_{i=0}^{n-1}\rho^i \right)_{n \geq 0}$ is increasing we have
$$
\frac{e_{23}}{e_{24}}\rho^{n+1}-\nu \sum_{i=0}^{n}\rho^i -\nu > \frac{e_{23}}{e_{24}}\rho -\nu >0
$$
and 
$$
\frac{e_{24}}{c_{25}}\left(\frac{e_{23}}{e_{24}}\rho^{n+1}-\nu \sum_{i=0}^{n}\rho^i\right) -\beta > \frac{e_{23}}{c_{25}}\rho -\beta = -\delta >0.
$$
Hence, by choosing $x_3 \ll 1$ we finish the proof.
\end{proof}

Note that the bigger the $n$, the smaller $x_3$ has to be. So, points whose dynamics are represented by long sequences of $L$'s are found very close to the $x_4$-axis. The choice of $x_3$ in this proof does not exclude the choice made in the proof of Proposition \ref{RLR_LRL-transition}.

\section{Concluding remarks}\label{conclusion}
This work deals with switching near simple networks in $\R^n$. We establish as a general result what several examples in the literature have been pointing at: that switching more complex than at a node is severely constrained unless the linearization of the vector field at a node has complex eigenvalues. We find a sufficient condition that prevents even switching along a connection, a strong constraint on the dynamics near the network. We consider an example where the sufficient condition is relaxed and study some of the dynamical features that arise from a new type of switching. Although infinite switching does not occur, several interesting possibilities are present in the dynamics near this network. An exhaustive description of the dynamics of this example is out of the scope of this work. We hope that our description serves as a useful illustration for further research in this area.

\paragraph{Acknowledgements:}
The authors were, in part, supported by CMUP (UID/MAT/00144/2013), funded by the Portuguese Government through the Funda\c{c}\~ao para a Ci\^encia e a Tecnologia (FCT) with national (MEC) and European structural funds through the programs FEDER, under the partnership agreement PT2020.
The second author further acknowledges support through FCT grant Incentivo/MAT/UI0144/2014.

\subsection*{A \hspace{.2cm} Local maps for the Bowtie network}

In order to construct the return map around the $R$-cycle 
$h_R: \; H_1^{\out,2} \rightarrow H_1^{\out,2}$ given by $h_R=\phi_{312} \circ \phi_{231} \circ \phi_{123}$, we provide the local maps. Analogously, for the return map around the $L$-cycle, $h_L: \; H_5^{\out,2} \rightarrow H_5^{\out,2}$, given by $h_L= \phi_{452} \circ \phi_{245} \circ \phi_{524}$.

Note that, because the global maps are the identity, we can identify $H_i^{\out,j}$ with $H_{j}^{\inn,i}$.
\bigbreak

\paragraph{For the $R$-cycle:}
\begin{eqnarray*}
\phi_{123}: H_{2}^{\inn,1} \to H_{2}^{\out,3}, \quad \phi_{123}(1,x_2,x_3,x_4,x_5) & = & \Big(x_3^{\frac{c_{21}}{e_{23}}}, x_2x_3^{\frac{r_2}{e_{23}}},1, x_4 x_3^{-\frac{e_{24}}{e_{23}}} , x_5 x_3^{\frac{c_{25}}{e_{23}}} \Big) \\
\phi_{231}: H_{3}^{\inn,2} \to H_{3}^{\out,1}, \quad \phi_{231}(x_1,1,x_3,x_4,x_5) & = & \Big(1, x_1^{\frac{c_{32}}{e_{31}}}, x_3x_1^{\frac{r_3}{e_{31}}}, x_4 x_1^{\frac{c_{34}}{e_{31}}}, x_5 x_1^{\frac{c_{35}}{e_{31}}} \Big) \\
\phi_{312}: H_{1}^{\inn,3} \to H_{1}^{\out,2}, \quad \phi_{312}(x_1,x_2,1,x_4,x_5) & = & \Big(x_1 x_2^{\frac{r_1}{e_{12}}}, 1, x_2^{\frac{c_{13}}{e_{12}}}, x_4x_2^{\frac{c_{14}}{e_{12}}}, x_5 x_2^{\frac{c_{15}}{e_{12}}} \Big)
\end{eqnarray*}

\paragraph{For the $L$-cycle:}
\begin{eqnarray*}
\phi_{524}: H_{2}^{\inn,5} \to H_{2}^{\out,4}, \quad \phi_{524}(x_1,x_2,x_3,x_4,1) & = & \Big(x_1 x_4^{\frac{c_{21}}{e_{24}}}, x_2x_4^{\frac{r_2}{e_{24}}}, x_3 x_4^{-\frac{e_{23}}{e_{24}}},1, x_4^{\frac{c_{25}}{e_{24}}}  \Big) \\
\phi_{245}: H_{4}^{\inn,2} \to H_{4}^{\out,5}, \quad \phi_{245}(x_1,1,x_3,x_4,x_5) & = & \Big(x_1x_5^{\frac{c_{41}}{e_{45}}}, x_5^{\frac{c_{42}}{e_{45}}}, x_3x_5^{\frac{c_{43}}{e_{45}}}, x_4 x_5^{\frac{r_4}{e_{45}}}, 1  \Big) \\
\phi_{452}: H_{5}^{\inn,4} \to H_{5}^{\out,2}, \quad \phi_{452}(x_1,x_2,x_3,1,x_5) & = & \Big(x_1 x_2^{\frac{c_{51}}{e_{52}}}, 1, x_3x_2^{\frac{c_{53}}{e_{52}}}, x_2^{\frac{c_{54}}{e_{52}}}, x_5 x_2^{\frac{r_5}{e_{52}}} \Big) 
\end{eqnarray*}

\paragraph{For the transition between cycles:}
\begin{eqnarray*}
\phi_{124}: H_{2}^{\inn,1} \to H_{2}^{\out,4}, \quad \phi_{124}(1,x_2,x_3,x_4,x_5) & = & \Big(x_4^{\frac{c_{21}}{e_{24}}}, x_2x_4^{\frac{r_2}{e_{24}}}, x_3 x_4^{-\frac{e_{23}}{e_{24}}},1, x_5 x_4^{\frac{c_{25}}{e_{24}}}  \Big) \\
\phi_{523}: H_{2}^{\inn,5} \to H_{2}^{\out,3}, \quad \phi_{523}(x_1, x_2, x_3, x_4, 1) & = & \Big( x_1x_3^{\frac{c_{21}}{e_{23}}}, x_2 x_3^{\frac{r_2}{e_{23}}}, 1, x_4 x_3^{-\frac{e_{24}}{e_{23}}}, x_3^{\frac{c_{25}}{e_{23}}} \Big)
\end{eqnarray*}

\paragraph{The return maps:} We select only the relevant components for the return maps, that is, we discard the radial and outgoing components. The return maps are:
\begin{eqnarray*}
h_R(x_3,x_4,x_5) & = & (x_3^{\tilde{\rho}}, x_4 x_3^{\tilde{\nu}}, x_5 x_3^{\tilde{\mu}}) \\
h_L(x_1,x_3,x_4) & = & (x_1 x_4^{\mu}, x_3 x_4^{\nu}, x_4^{\rho})
\end{eqnarray*}
Iteration of the return maps leads to:
\begin{eqnarray*}
h_R^n(x_3,x_4,x_5) & = & \left(x_3^{\tilde{\rho}^n}, x_4 x_3^{\tilde{\nu} \sum_{i=0}^{n-1}\tilde{\rho}^i}, x_5 x_3^{\tilde{\mu} \sum_{i=0}^{n-1}\tilde{\rho}^i}\right) \\
h_L^n(x_1,x_3,x_4) & = & \left(x_1 x_4^{\mu \sum_{i=0}^{n-1}\rho^i}, x_3 x_4^{\nu \sum_{i=0}^{n-1}\rho^i}, x_4^{\rho^n}\right)
\end{eqnarray*}
Note that, using all coordinates, $D_{h_R}=D_{\phi_{123}} \subset H_1^{\out,2}$ and $D_{h_L}=D_{\phi_{524}} \subset H_5^{\out,2}$.

\end{document}